\documentclass[11pt,leqno,a4paper]{amsart}
\usepackage{amsmath,amssymb,amscd,tikz-cd}
\usepackage[hidelinks]{hyperref}
\hypersetup{pdftitle={Factorisation and matrix amplification of real rank zero inclusions},pdfauthor={Jamie Bell}}

\newtheorem{theorem}{Theorem}[section]
\newtheorem{corollary}[theorem]{Corollary}
\newtheorem{proposition}[theorem]{Proposition}
\newtheorem{lemma}[theorem]{Lemma}
\theoremstyle{definition}
\newtheorem{example}[theorem]{Example}
\theoremstyle{plain}
\newtheorem{maintheorem}{Theorem}

\newtheorem{question*}{Question}

\newcommand{\C}{\mathbb C}
\newcommand{\R}{\mathbb R}
\newcommand{\Z}{\mathbb Z}

\newcommand{\K}{\mathcal K}
\newcommand{\sa}{\mathrm{sa}}

\newcommand{\rr}{\operatorname{rr}}
\newcommand{\Her}{\operatorname{Her}}
\newcommand{\GL}{\operatorname{GL}}
\newcommand{\dist}{\operatorname{dist}}
\newcommand{\diag}{\operatorname{diag}}

\newcommand{\interior}{\operatorname{int}}
\newcommand{\norm}[1]{\lVert#1\rVert}

\title[Factorisation of real rank zero inclusions]{Factorisation and matrix amplification of real rank zero inclusions}
\author{Jamie Bell}
\address[Jamie Bell]{Mathematical Institute, University of M\"unster, Einsteinstr.\ 62, 48149 M\"unster, Germany}
\email{jbell@uni-muenster.de}
\thanks{Funded by the Deutsche Forschungsgemeinschaft (DFG, German Research
Foundation) under Germany's Excellence Strategy EXC 2044/2 -390685587,
Mathematics M\"unster: Dynamics--Geometry--Structure and project-ID
427320536, SFB 1442, of the DFG}

\begin{document}

\begin{abstract}
We show that real rank zero is not preserved under matrix amplification of inclusions. This yields real rank zero inclusions that do not approximately factor through real rank zero $C^*$-algebras,
thereby resolving two questions of Gabe and Neagu. 
\end{abstract}

\maketitle

\section{Introduction}

Introduced by Brown and Pedersen in 1991, real rank zero is an important regularity property which expresses an abundance of projections in a $C^*$-algebra. More specifically, it means that every hereditary $C^*$-subalgebra has an approximate unit of projections
\cite{BP}. Gabe and Neagu extended this notion to inclusions of
$C^*$-algebras, and more generally to homomorphisms \cite{GN}.
For an inclusion $A\subseteq B$, real rank zero means that the
hereditary $C^*$-subalgebra of $B$ generated by any nonzero positive
element of $A$ has an approximate unit of projections. This notion is motivated, at least in part, by the classification of $C^*$-algebras -- particularly the Elliott classification programme -- where the most definitive results are generally obtained by first classifying maps between $C^*$-algebras in the form of existence and uniqueness theorems and then, via an intertwining argument, deducing classification for $C^*$-algebras themselves. One emerging theme in this endeavour is that regularity of the domain or codomain $C^*$-algebra can often be transferred to the maps themselves, leading to regularity notions of morphisms, which often afford greater flexibility \cite{GabeO2,BGSW,CGSTW,GabeOinf,CN,BV,GHP}.

It is not difficult to see that any inclusion that (approximately) factors through a real
rank zero $C^*$-algebra has real rank zero. For unital inclusions of commutative
$C^*$-algebras, or suitable full nuclear $\mathcal O_\infty$-stable maps,
Gabe and Neagu prove the converse \cite[Theorems 2.1 and 5.5]{GN}. They ask whether such a factorisation characterises real rank zero inclusions in general. We show that the answer is negative, even for inclusions with commutative domains. Let $D^3$ be the closed unit ball in $\R^3$. Our main result is the following. 

\begin{maintheorem}\label{thm:A}
There exists a separable unital $C^*$-algebra $B$ and a unital inclusion $C(D^3)\subseteq B$ which has real rank zero but does not approximately factor through real rank zero $C^*$-algebras.
\end{maintheorem}

The proof of Theorem~\ref{thm:A} in fact yields a stronger conclusion. There are real rank zero inclusions that cannot be approximated by compositions of two real rank zero homomorphisms, even up
to approximate Murray--von Neumann equivalence (Corollary~\ref{cor:factor}). The obstruction comes
from matrix amplification of maps. Since real rank zero of algebras passes to
matrix amplifications, an inclusion approximately factoring through real
rank zero algebras also has real rank zero at every matrix level. Here we
construct real rank zero inclusions for which real rank zero already fails
at the $2\times2$ level, which resolves \cite[Question 4.4]{GN} in the
negative.

\begin{maintheorem}\label{thm:B}
Let $J$ be a nonunital separable simple $C^*$-algebra with real rank zero
and continuous scale. Suppose that there exists a nondegenerate
monomorphism $\varphi: C_0(\interior D^3)\to J$ with
$K_1(\varphi)\ne0$. Then $\varphi$ extends to a unital inclusion
$\iota: C(D^3)\to M(J)$ of real rank zero such that $\iota^{(2)}$ does
not have real rank zero. Moreover, there is a separable unital
$C^*$-subalgebra $B\subseteq M(J)$ containing $\iota(C(D^3))$ such
that the corestriction $\iota\colon C(D^3)\to B$ has the same properties.
\end{maintheorem}

The criterion used to establish Theorem~\ref{thm:B} applies to extensions
by real rank zero ideals with purely infinite simple quotients (Proposition~\ref{prop:criterion}). This is applied to $M(J)$, considered as an extension of the corona algebra $Q(J)$ by $J$. In this case, we show that a unital inclusion $A\subseteq M(J)$ has real rank zero exactly when the exponential
map in $K$-theory vanishes on the projection classes in the image of $A$
in the corona. Matrix amplifications test additional projection classes
which may not be represented in the algebra itself. For $D^3$ and its
boundary $S^2$, this difference is detected by the Bott projection. The
scalar projections in $C(S^2)$ are trivial and therefore lift, while the
Bott projection in $M_2(C(S^2))$ corresponding to the Hopf bundle
has a nonzero boundary class. The hypothesis on $K_1(\varphi)$ ensures
that this obstruction is preserved in the codomain.

The paper is organised as follows. We recall some preliminaries in Section~\ref{sec:prelim}, before establishing the matrix amplification criterion and proving Theorem~\ref{thm:B} in Section~\ref{sec:matrix}. Section~\ref{sec:factor} gives the factorisation obstruction. In
Section~\ref{sec:examples}, we construct explicit examples and complete
the proof of Theorem~\ref{thm:A}.

\medskip 
\noindent\emph{Acknowledgements}. I am grateful to Wilhelm Winter for useful discussions, and for asking a question which led to the current formulation of our main results.

\medskip 
\noindent\emph{AI Declaration}. ChatGPT-6 Pro was used during the preparation of this work, including identifying the Bott class obstruction, developing a first draft of the manuscript, literature searches, and proofreading. The author takes full responsibility for the contents and validity of the paper.

\section{Preliminaries}\label{sec:prelim}

All homomorphisms are $*$-homomorphisms and are not assumed unital unless
stated otherwise. Unital inclusions preserve the unit. For a
$C^*$-algebra $A$, write $A_{\sa}$ for its self-adjoint part and
$\Her_A(a)=\overline{aAa}$ for the hereditary $C^*$-subalgebra generated
by $a\in A_+$. For a homomorphism $\theta\colon A\to B$, write
$\theta^{(n)}\colon M_n(A)\to M_n(B)$ for its entrywise amplification.

\subsection{Real rank zero and factorisation}
A homomorphism $\theta: A\to B$ has \emph{real rank zero} if
$\Her_B(\theta(a))$ has an approximate unit of projections for every
$a\in A_+$. Equivalently, every element of $\theta(A)_{\sa}$ can be
approximated by finite-spectrum self-adjoint elements of $B$
\cite[Theorems 1.6 and 1.7]{GN}. If $A$, $B$, and $\theta$ are unital,
this is also equivalent to approximation by invertible self-adjoints in
$B$. Thus, writing
\[
 d_B(x)=\dist(x,\GL_{\sa}(B)),\quad x\in B_{\sa},
\]
a unital homomorphism $\theta$ has real rank zero exactly when $d_B(\theta(a))=0$ for every $a\in A_{\sa}$. By the triangle inequality, $d_B$ is $1$-Lipschitz, i.e.\ $|d_B(x) - d_B(y)| \le \|x-y\|$ for all $x,y\in B_{\sa}$. We shall use that real rank zero for $C^*$-algebras passes to hereditary subalgebras, quotients, and matrix algebras \cite{BP}.

We say that $\theta: A\to B$ \emph{approximately factors through
real rank zero algebras} if, for every finite $\mathcal F\subseteq A$ and
$\varepsilon>0$, there are homomorphisms $\alpha: A\to C$ and
$\beta: C\to B$, with $\rr(C)=0$, such that for all $a\in \mathcal F$, 
\[
 \norm{\theta(a)-\beta\alpha(a)}<\varepsilon.
\]
The intermediate algebra and the maps may vary with $\mathcal F$ and
$\varepsilon$, and need not be unital. Following \cite[Definition 3.4 and Lemma 3.5]{GabeO2}, homomorphisms $\theta,\eta: A\to B$ are \emph{approximately Murray--von Neumann
equivalent} if there is a net of contractions $(v_\lambda)$ in $B$ such
that for all $a\in A$, 
\begin{equation}\label{eq:mvnequiv}
 \|v_\lambda^*\theta(a)v_\lambda - \eta(a)\| \to 0,\quad \text{and} \quad
 \|v_\lambda\eta(a)v_\lambda^* - \theta(a)\| \to 0.
\end{equation}

\subsection{Extensions and continuous scale}

Recall the well-known fact that if $ 0\rightarrow J\rightarrow A\rightarrow B\rightarrow0$ is an extension with $\rr(J)=\rr(B)=0$, then $\rr(A)=0$ if and only if the exponential boundary map $K_0(B)\to K_1(J)$ vanishes (see \cite[Proposition 4]{LR} and \cite[Proposition 2.5]{Thiel}). In particular, this applies when $B$ is finite-dimensional.

Zhang proved that simple purely infinite $C^*$-algebras have real rank zero \cite{Zhang}. We also use the following consequence of Cuntz's projection theory \cite{Cuntz}.

\begin{lemma}\label{lem:comp}
	Let $A$ be a simple purely infinite $C^*$-algebra, and let $p,q \in A$ be projections with $q\ne 0$. Then there is a projection $0 < r < q$ such that $[r] = -[p]$ in $K_0(A)$. 
\end{lemma}

\begin{proof}
Since $q$ is infinite, choose a nonzero projection $s<q$. The corner $sAs$ is simple purely infinite, with unit $s$. The closed ideal of $A$ generated by $s$ is nonzero and hence equals $A$, so $sAs$ is a full corner. By Morita invariance, the inclusion $\iota\colon sAs\hookrightarrow A$ induces an isomorphism $\iota_*\colon K_0(sAs)\to K_0(A)$. By \cite[Theorem 6.11.7]{Blackadar}, every class in $K_0(sAs)$ is represented by a nonzero projection in $sAs$. Thus there is a nonzero projection $r\in sAs$ with $[r]_{sAs}=\iota_*^{-1}(-[p]_A)$. Since $s$ is the unit of $sAs$, we have $0<r\leq s<q$, and
$[r]_A=\iota_*([r]_{sAs})=-[p]_A$.
\end{proof}

A nonunital $\sigma$-unital simple $C^*$-algebra $J$ has \emph{continuous scale} if it has a sequential approximate unit $(e_n)$ of positive contractions with $e_{n+1}e_n=e_n$ for all $n$ such that, for every nonzero $b\in J_+$, there is $N\ge 1$ such that for all $m > n\ge N$,
\[
 e_m-e_n\precsim b.
\]
Here $\precsim$ denotes Cuntz subequivalence, which for projections agrees with Murray--von Neumann subequivalence. For nonelementary $J$, continuous scale is equivalent to the corona algebra $Q(J) = M(J)/J$ being simple purely infinite \cite{Lin91,Lin04}; see also \cite[Theorem 1.2]{Ng}.

\section{Failure of matrix amplification}\label{sec:matrix}


\begin{lemma}\label{lem:gap}
Let $J$ be a closed ideal in a unital $C^*$-algebra $E$ with $\rr(J) = 0$ and $B = E/J$ simple purely infinite. Let $\pi : E \to B$ be the quotient map and let $\partial : K_0(B) \to K_1(J)$ be the exponential map. For $x\in E_{\sa}$, put $b = \pi(x)$. If $0\in \sigma(b)$, then $d_E(x) = 0$. If $0\not\in \sigma(b)$, then 
\[
	d_E(x) = 0 \iff \partial([1_{(0,\infty)}(b)]) = 0.
\] 	
\end{lemma}

\begin{proof}
For a projection $p\in B$, put $F_p=\pi^{-1}(C^*(1_B,p))$.
Since $C^*(1_B,p)$ is finite-dimensional, its $K_0$-group is generated
by $[p]$ and $[1_B-p]$. Moreover, $\partial([1_B])=0$, since $1_B$
lifts to $1_E$. Thus naturality of the boundary map and
\cite[Proposition 4]{LR} (see also \cite[Proposition 2.5]{Thiel}) give
\[
    \rr(F_p)=0 \iff \partial([p])=0.
\]

Suppose first that $0\notin\sigma(b)$. The function $1_{(0,\infty)}$
is continuous on $\sigma(b)$, so continuous functional calculus gives
a projection
\[
    p=1_{(0,\infty)}(b)\in C^*(1_B,b)\subseteq B.
\]
If $\partial([p])=0$, choose a self-adjoint lift $\ell\in E$ of
$\log|b|$ and set $v=\exp(-\ell/2)$. Then $v$ is positive and
invertible, and
\[
    \pi(vxv)=|b|^{-1/2}b|b|^{-1/2}=2p-1_B.
\]
Hence $vxv\in F_p$, which has real rank zero by the preceding
application of the extension theorem. Choose invertible self-adjoint
elements $z_n\in F_p$ converging to $vxv$. Then
$v^{-1}z_nv^{-1}$ are invertible self-adjoint elements of $E$
converging to $x$, so $d_E(x)=0$.

Conversely, suppose that $d_E(x)=0$. Choose an invertible
self-adjoint $y\in E$ with
\[
    \|y-x\|<\|b^{-1}\|^{-1}.
\]
The path $b_t=b+t(\pi(y)-b)$, $0\leq t\leq1$, has a uniform
spectral gap at zero. Continuous functional calculus therefore gives
a norm-continuous path of projections $1_{(0,\infty)}(b_t)$ from
$p$ to $1_{(0,\infty)}(\pi(y))$. The latter is the image of the
projection $1_{(0,\infty)}(y)\in E$. Homotopy invariance and exactness
of the six-term sequence yield $\partial([p])=0$.

Now suppose that $0\in\sigma(b)$, and fix $\delta>0$. Since $B$ has
real rank zero, choose a finite-spectrum invertible self-adjoint
element
\[
    c=\sum_{j=1}^m\lambda_jq_j,
    \qquad \|c-b\|<\delta,
\]
where the $\lambda_j$ are distinct nonzero real numbers and the
$q_j\in C^*(1_B,c)$ are nonzero orthogonal spectral projections
summing to $1_B$. Some $\lambda_i$ satisfies $|\lambda_i|<\delta$;
otherwise $\|c^{-1}\|\|c-b\|<1$ would imply that $b$ is invertible.

Put $q=q_i$ and
\[
    p=\sum_{\substack{j\ne i\\ \lambda_j>0}}q_j.
\]
By Lemma~\ref{lem:comp}, there is a projection $0<r<q$ such that
$[r]=-[p]$ in $K_0(B)$. Set
\[
    b'=\sum_{j\ne i}\lambda_jq_j+\delta r-\delta(q-r).
\]
Then $b'$ is invertible and self-adjoint, and its positive spectral
projection is $p+r$. In particular, $\partial([1_{(0,\infty)}(b')])
    =\partial([p]+[r])=0$. Moreover,
\[
    \|b'-b\|
    \leq \|c-b\|
       +\|\delta r-\delta(q-r)-\lambda_iq\|
    <3\delta.
\]
Choose a self-adjoint lift $t\in E$ of $b'-b$ with $\|t\|<4\delta$.
Since $\pi(x+t)=b'$, the invertible case gives $d_E(x+t)=0$.
The $1$-Lipschitz property of $d_E$ now yields
\[
    d_E(x)\leq d_E(x+t)+\|t\|<4\delta.
\]
As $\delta>0$ was arbitrary, $d_E(x)=0$.
\end{proof}

\begin{proposition}\label{prop:criterion}
Let $J$ be a closed ideal in a unital $C^*$-algebra $E$ with $\rr(J) = 0$ and $B = E/J$ simple purely infinite. Let $\partial : K_0(B) \to K_1(J)$ be the exponential map, and let $\pi : E \to B$ be the quotient map. For every unital $C^*$-subalgebra $A\subseteq E$ and $n\ge 1$, the inclusion $M_n(A)\subseteq M_n(E)$ has real rank zero if and only if $\partial([p]) = 0$ for every projection $p \in M_n(\pi(A))$. 
\end{proposition}

\begin{proof}
For $n=1$, suppose that $A\subseteq E$ has real rank zero and let $p\in\pi(A)$ be a projection. Choose $a\in A_{\sa}$ with $\pi(a)=2p-1$. Then $d_E(a)=0$, so Lemma~\ref{lem:gap} gives $\partial([p])=0$. Conversely, suppose $\partial$ vanishes on every projection in $\pi(A)$. For $a\in A_{\sa}$, if $\pi(a)$ is invertible, its positive spectral projection belongs to $\pi(A)$ and has zero boundary. Whether or not $\pi(a)$ is invertible, Lemma~\ref{lem:gap} gives $d_E(a)=0$, so the inclusion has real rank zero. The general case follows by applying the previous case to the extension $0\to M_n(J)\to M_n(E) \to M_n(B) \to 0$. Under the canonical identifications $K_0(M_n(B))\cong K_0(B)$ and $K_1(M_n(J))\cong K_1(J)$, the exponential map for this extension is $\partial$. 
\end{proof}

\begin{proof}[Proof of Theorem~\ref{thm:B}]
Put $I=C_0(\interior D^3)$, and regard $C(D^3)$ as a unital $C^*$-subalgebra
of $M(I)$ by restriction to the interior. Nondegeneracy gives a unital
extension $\overline\varphi: M(I)\to M(J)$ which is injective. Set $\iota=\overline\varphi|_{C(D^3)}$.

We have $\iota^{-1}(J)=I$. Indeed, let $(u_k)$ be an approximate unit
of $I$. Then $(\varphi(u_k))$ is an approximate unit of $J$. If
$\iota(f)\in J$, injectivity gives
\[
 \norm{f-fu_k}=\norm{\iota(f)-\iota(f)\varphi(u_k)}\longrightarrow0,
\]
so $f\in I$. The reverse containment is immediate. Thus $\iota$
induces a unital monomorphism $\theta: C(S^2)\to Q(J)$ and a
commutative diagram
\begin{equation}\label{eq:extensions}
\begin{tikzcd}
	0 & I & {C(D^3)} & {C(S^2)} & 0 \\
	0 & J & {M(J)} & {Q(J)} & 0
	\arrow[from=1-1, to=1-2]
	\arrow[from=1-2, to=1-3]
	\arrow["\varphi"', from=1-2, to=2-2]
	\arrow[from=1-3, to=1-4]
	\arrow["\iota"', from=1-3, to=2-3]
	\arrow[from=1-4, to=1-5]
	\arrow["\theta"', from=1-4, to=2-4]
	\arrow[from=2-1, to=2-2]
	\arrow[from=2-2, to=2-3]
	\arrow[from=2-3, to=2-4]
	\arrow[from=2-4, to=2-5]
\end{tikzcd}
\end{equation}
Since $K_1(\varphi)\ne0$, the algebra $J$ is nonelementary. Its
continuous scale therefore makes $Q(J)$ purely infinite simple.
Connectedness of $S^2$ implies that the only projections in $C(S^2)$
are $0$ and $1$, both of which lift in \eqref{eq:extensions}.
Proposition~\ref{prop:criterion} gives real rank zero of $\iota$.

Let $x_1,x_2,x_3$ be the coordinate functions on $S^2$, and consider
the Bott projection \cite[Example 6.2.3]{Rosenberg}
\[
    p=\frac12
    \begin{pmatrix}
        1+x_1 & x_2+ix_3\\
        x_2-ix_3 & 1-x_1
    \end{pmatrix}
    \in M_2(C(S^2)).
\]
The classes $[1]$ and $[p]$ form a basis of
$K_0(C(S^2))\cong\mathbb Z^2$. Write $\partial_I$ and $\partial_J$
for the exponential boundary maps associated to the two rows of
\eqref{eq:extensions}. Since $D^3$ is contractible, the restriction
map $K_0(C(D^3))\to K_0(C(S^2))$ has image $\mathbb Z[1]$, and
$K_1(C(D^3))=0$. Exactness therefore shows that $\partial_I([p])$
is a generator of $K_1(I)\cong\mathbb Z$. Since $K_1(\varphi)\ne0$,
naturality gives
\begin{equation}\label{eq:bott-boundary}
    \partial_J([\theta^{(2)}(p)])
    =K_1(\varphi)(\partial_I([p]))\ne0.
\end{equation}
Proposition~\ref{prop:criterion} now shows that $\iota^{(2)}$ does not
have real rank zero.

Finally, the codomain can be made separable. Choose a dense sequence
$(a_m)$ in $\iota(C(D^3))_{\sa}$ and finite-spectrum self-adjoints
$b_{m,k}\in M(J)$ with $\|a_m-b_{m,k}\|<1/k$. Set
\begin{equation}\label{eq:B}
    B=C^*\bigl(J,\iota(C(D^3)),b_{m,k}:m,k\geq1\bigr)
    \subseteq M(J).
\end{equation}
This algebra is separable and unital. Spectral permanence and density
give real rank zero of $\iota\colon C(D^3)\to B$. Its second
amplification cannot have real rank zero, since finite-spectrum
approximants in $M_2(B)$ would also be such approximants in
$M_2(M(J))$.
\end{proof}

\section{The factorisation obstruction}\label{sec:factor}

\begin{lemma}\label{lem:Schur}
Let $A,B,C$ be $C^*$-algebras, with $A$ and $B$ unital, and let $\alpha : A \to C$ and $\beta : C \to B$ be homomorphisms of real rank zero. If $\beta\alpha$ is unital, then $(\beta\alpha)^{(2)}$ has real rank zero.  
\end{lemma}

\begin{proof}
Put $e=\alpha(1_A)$. Then $\beta(e)=1_B$, and the maps
$A\to eCe\to B$ are unital and still have real rank zero. For the
first map this follows from
$\Her_{eCe}(\alpha(a))=\Her_C(\alpha(a))$ for $a\in A_+$; the second
is a restriction of $\beta$. We may therefore assume $C$, $\alpha$,
and $\beta$ are unital.

Given
\[
 X=\begin{pmatrix}a&b\\b^*&c\end{pmatrix}\in M_2(A)_{\sa}
 \quad\text{and}\quad\varepsilon>0,
\]
choose $h\in\GL_{\sa}(C)$ such that $\norm{h-\alpha(a)}<\varepsilon$.
Write $b_0=\alpha(b)$ and $u=\alpha(c)-b_0^*h^{-1}b_0$, and choose
$k\in\GL_{\sa}(B)$ with $\norm{k-\beta(u)}<\varepsilon$.
Writing $H=\beta(h)$ and $D=\beta(b_0)$, the factorisation
\[
 \begin{pmatrix}H&D\\D^*&k+D^*H^{-1}D\end{pmatrix}
 =\begin{pmatrix}1&0\\D^*H^{-1}&1\end{pmatrix}
  \begin{pmatrix}H&0\\0&k\end{pmatrix}
  \begin{pmatrix}1&H^{-1}D\\0&1\end{pmatrix}
\]
shows that the matrix on the left is invertible and self-adjoint. Its
difference from $(\beta\alpha)^{(2)}(X)$ is $\diag(\beta(h-\alpha(a)),\,k-\beta(u))$,
which has norm less than $\varepsilon$.
\end{proof}

\begin{corollary}\label{cor:factor}
Let $A$ and $B$ be unital $C^*$-algebras and let $\iota\colon A\to B$
be a unital homomorphism. If $\iota^{(2)}$ does not have real rank zero,
then $\iota$ is neither a point-norm limit of compositions of two real
rank zero homomorphisms nor approximately Murray--von Neumann
equivalent to such a composition.
\end{corollary}

\begin{proof}
Suppose $\eta_\lambda=\beta_\lambda\alpha_\lambda\to\iota$ pointwise,
where $\alpha_\lambda\colon A\to C_\lambda$ and
$\beta_\lambda\colon C_\lambda\to B$ have real rank zero. The projections
$\eta_\lambda(1_A)$ converge to $1_B$, so they are eventually equal to
$1_B$. Lemma~\ref{lem:Schur} implies real rank zero of
$\eta_\lambda^{(2)}$ for large $\lambda$. These maps converge pointwise
to $\iota^{(2)}$, which therefore has real rank zero by
\cite[Lemma 1.10]{GN}, a contradiction.

For the second assertion, let $\alpha\colon A\to C$ and
$\beta\colon C\to B$ have real rank zero, and suppose
$\eta=\beta\alpha$ is approximately Murray--von Neumann equivalent to
$\iota$, implemented by $(v_\lambda)$ as in \eqref{eq:mvnequiv}.
Put $q=\eta(1_A)$ and $e=\alpha(1_A)$. Evaluating at the unit gives
$v_\lambda^*v_\lambda\to q$ and $v_\lambda qv_\lambda^*\to1_B$ in
norm. Let $x_\lambda=v_\lambda q$. For sufficiently large $\lambda$,
$x_\lambda^*x_\lambda$ is invertible in $qBq$ and
$x_\lambda x_\lambda^*$ is invertible in $B$. Thus
$w_\lambda=x_\lambda(x_\lambda^*x_\lambda)^{-1/2}$ satisfies
$w_\lambda^*w_\lambda=q$, $w_\lambda w_\lambda^*=1_B$, and
$\norm{w_\lambda-x_\lambda}\to0$. Since $q\eta(a)q=\eta(a)$, the
unital homomorphisms
\[
    \rho_\lambda\colon A\to B,\qquad
    \rho_\lambda(a)=w_\lambda\eta(a)w_\lambda^*,
\]
converge pointwise to $\iota$. Each $\rho_\lambda$ factors as
$\beta_\lambda\alpha_e$, where
\[
    \alpha_e\colon A\to eCe,\quad \alpha_e(a)=\alpha(a),
    \qquad
    \beta_\lambda\colon eCe\to B,\quad
    \beta_\lambda(c)=w_\lambda\beta(c)w_\lambda^*.
\]
The map $\alpha_e$ is the corestriction of $\alpha$ and has real rank
zero as in Lemma~\ref{lem:Schur}. The map $\beta_\lambda$ also has
real rank zero: restrict $\beta$ to $eCe\to qBq$ and compose with
the isomorphism $qBq\to B$ implemented by $w_\lambda$. Both maps are
unital, so this contradicts the first assertion.
\end{proof}

\section{Examples}\label{sec:examples}

Throughout this section, put $I=C_0(\interior D^3)\cong C_0(\R^3)$. The following lemma
ensures nondegeneracy in both constructions.

\begin{lemma}\label{lem:hereditary}
Let $A$ be a separable simple $C^*$-algebra of real rank zero, let
$\gamma\colon I\to A$ be a monomorphism, and let $h\in I_+$ be
strictly positive. Then $J=\Her_A(\gamma(h))$ is nonunital,
separable, simple, and of real rank zero. Corestricting $\gamma$
gives a nondegenerate monomorphism $I\to J$.
\end{lemma}

\begin{proof}
Set $f_n(t)=t/(t+1/n)$ and $u_n=f_n(h)$. Then $(u_n)$ is an
approximate unit for $I$, and $(\gamma(u_n))=(f_n(\gamma(h)))$ is an
approximate unit for $J$. For $f\in I$, the elements
$\gamma(u_nfu_n)$ belong to $J$ and converge to $\gamma(f)$.
Thus $\gamma(I)\subseteq J$, and the corestriction is nondegenerate.
If $J$ were unital, $(\gamma(u_n))$ would converge in norm to its
unit. Since $\gamma$ is isometric, $(u_n)$ would then converge in
norm to a unit for $I$, a contradiction.

As a nonzero hereditary subalgebra of $A$, the algebra $J$ is
separable, simple, and of real rank zero \cite{BP}.
\end{proof}

\subsection{Purely infinite ideals}
A Kirchberg algebra is a separable nuclear purely infinite simple
$C^*$-algebra. For a stable Kirchberg algebra $J$, real rank zero
and continuous scale are automatic. The remaining hypothesis of
Theorem~\ref{thm:B} is therefore an embedding of $I$ which detects a
nonzero class in $K_1(J)$. We obtain such an embedding by first
specifying its $KK$-class and then applying an existence theorem
for homomorphisms.

\begin{proposition}\label{prop:kirchberg}
Let $J$ be a stable Kirchberg algebra with $K_1(J)\ne0$.
Then $J$ has real rank zero and continuous scale, and there is a
nondegenerate monomorphism $\varphi\colon I\to J$ with
$K_1(\varphi)\ne0$.
\end{proposition}

\begin{proof}
The algebra $J$ has real rank zero by \cite{Zhang}. Choose an
increasing approximate unit of projections $(e_n)$ for $J$.
Pure infiniteness and simplicity give $e_m-e_n\precsim b$ whenever
$m>n$ and $0\ne b\in J_+$, so $J$ has continuous scale.

Let $b\in K_1(I)=KK^1(\C,I)$ be the odd Bott class. Since
$I\cong C_0(\R^3)$, this class is a graded $KK$-equivalence:
there is $d\in KK^1(I,\C)$ with
$b\otimes_I d=1_{\C}$ and $d\otimes_{\C}b=1_I$
\cite[Theorem 19.2.1 and Section 19.2.5]{Blackadar}.
Choose $0\ne g\in K_1(J)=KK^1(\C,J)$ and set
\[
    \kappa=d\otimes_{\C}g\in KK(I,J).
\]
The map on $K_1$ induced by $\kappa$ sends the generator $b$ to
\[
    b\otimes_I\kappa
    =(b\otimes_I d)\otimes_{\C}g=g.
\]
Since $I$ is separable and nuclear, and $J$ is $\sigma$-unital and
contains a full properly infinite projection, Gabe's existence
theorem \cite[Theorem A]{GabeOinf} yields a full homomorphism
$\gamma\colon I\to J$ with $KK(\gamma)=\kappa$, where we use nuclearity of $I$ to identify $KK_{\mathrm{nuc}}(I,J)=KK(I,J)$. Fullness here means that the image of every nonzero positive element generates $J$ as a closed ideal; in particular, $\gamma$ is injective. Moreover, $K_1(\gamma)(b)=g$.

It remains to arrange nondegeneracy. Choose strictly positive
$h\in I_+$ and put $J_0=\Her_J(\gamma(h))$. By
Lemma~\ref{lem:hereditary}, the corestriction
$\gamma_0\colon I\to J_0$ is nondegenerate and $J_0$ is nonunital.
Writing $j\colon J_0\hookrightarrow J$ for the inclusion, we have
$K_1(j)K_1(\gamma_0)=K_1(\gamma)\ne0$.
The algebra $J_0$ is also separable and purely infinite simple, so
Zhang's dichotomy \cite[Theorem 1.2]{Zhang} makes it stable.
Since $J_0$ is full in $J$, Brown's stable isomorphism theorem
\cite[Theorem 2.8]{Brown} gives $J_0\otimes\K\cong J\otimes\K$.
Stability therefore gives an isomorphism $\psi\colon J_0\to J$,
and $\varphi=\psi\gamma_0$ has the required properties.
\end{proof}

\begin{example}
For a concrete choice, let
\[
 R=\begin{pmatrix}1&1&0\\1&1&1\\0&1&1\end{pmatrix}.
\]
This is an irreducible nonpermutation matrix, so the Cuntz--Krieger
algebra $\mathcal O_R$ is a unital Kirchberg algebra \cite{CK}.
The calculation in \cite{Cuntz} gives
\[
 K_1(\mathcal O_R)=\ker(1-R^{\mathsf t})=\Z(1,0,-1).
\]
Theorem~\ref{thm:B} applies with $J=\mathcal O_R\otimes\K$, by Proposition~\ref{prop:kirchberg}.
\end{example}

\subsection{Stably finite codomains}

Goodearl's construction \cite{Goodearl} provides an alternative in
which the multiplier algebra is stably finite. Its connecting maps
combine an identity summand with point evaluations. The identity
summand preserves the class in $K_1$, even as its proportion of the
matrix size tends to zero. Dense point evaluations ensure simplicity,
while the vanishing proportion of the identity summand gives real
rank zero.

\begin{proposition}\label{prop:Goodearl}
There exist a nonunital separable simple $C^*$-algebra $J$ of real
rank zero and continuous scale and a nondegenerate monomorphism
$\varphi\colon I\to J$ with $K_1(\varphi)\ne0$ such that $M(J)$
is stably finite.
\end{proposition}

\begin{proof}
Choose $\infty\in S^3$ and a sequence $(z_n)\subseteq S^3\setminus
\{\infty\}$ such that every tail is dense in $S^3$. Define
\[
 A_n=M_{2^{n-1}}(C(S^3)),\qquad
 \alpha_n(f)(z)=\diag(f(z),f(z_n)),
\]
and let $A=\varinjlim(A_n,\alpha_n)$, with first-stage map
$\eta\colon C(S^3)\to A$. The connecting maps are injective.
The density assumption gives simplicity by \cite[Lemma 1]{Goodearl}.
The proportion of the matrix size occupied by the identity summand
from stage $n$ to stage $m$ is $2^{n-m}$, which tends to zero as
$m\to\infty$. Thus $A$ has real rank zero by
\cite[Theorem 9]{Goodearl}.

We shall use the trace and projection comparison in $A$. Tracial
states exist since $A$ is a unital inductive limit of homogeneous
algebras. As $S^3$ is connected, each projection in $A_n$ has constant
rank, so all tracial states of $A$ agree on projections from every
stage. Every projection in $A$ is equivalent to one from a stage;
hence real rank zero implies that $A$ has a unique tracial state
$\tau$. It is faithful by simplicity. The comparison result in
\cite[p.~648, following Theorem 10]{Goodearl} gives, for all projections $p,q\in A$,
\begin{equation}\label{eq:trace-comparison}
    \tau(p)<\tau(q)\quad\Longrightarrow\quad p\precsim q.
\end{equation}

Identify $I$ with $C_0(S^3\setminus\{\infty\})$, choose strictly
positive $h\in I_+$, and put
\[
    J=\Her_A(\eta(h)),\qquad
    \varphi=\eta|_I\colon I\to J.
\]
Lemma~\ref{lem:hereditary} shows that $J$ is nonunital, separable,
simple, and of real rank zero, and that $\varphi$ is nondegenerate.
To check its $K_1$-map, note that point evaluation factors through a
matrix algebra over $\C$ and therefore induces the zero map on $K_1$.
Consequently, each $K_1(\alpha_n)$ is the identity under the standard
matrix identifications, and $K_1(\eta)$ is an isomorphism.
The inclusion $I\hookrightarrow C(S^3)$ also induces an isomorphism
on $K_1$, by the split extension associated to evaluation at
$\infty$. Since $K_1(I)\cong\Z$, the composite
\[
    K_1(I)\xrightarrow{K_1(\varphi)}K_1(J)
    \longrightarrow K_1(A)
\]
is nonzero, so $K_1(\varphi)\ne0$.

For continuous scale, choose an increasing approximate unit of
projections $(e_n)$ for $J$, and put $s=\norm{\tau|_J}$.
Given $0\ne b\in J_+$, real rank zero provides a nonzero projection
$q\in\Her_J(b)$. Since $\tau(e_n)\to s$ and $\tau(q)>0$, there is
$N$ such that, for all $m>n\geq N$,
\[
    \tau(e_m-e_n)\leq s-\tau(e_n)<\tau(q).
\]
Equation~\eqref{eq:trace-comparison} gives $e_m-e_n\precsim q$.
Both projections belong to the hereditary subalgebra $J$, so the
comparison takes place in $J$. As $q\precsim b$, this proves
continuous scale.

Finally, the faithful bounded trace $\tau|_J$ extends to a bounded
trace $\overline\tau$ on $M(J)$. To see that the extension is
faithful, suppose $x\in M(J)_+$ and $\overline\tau(x)=0$.
For $a\in J$, traciality gives
\[
    0\leq\tau(a^*xa)
    =\overline\tau(x^{1/2}aa^*x^{1/2})
    \leq\norm{a}^2\overline\tau(x)=0.
\]
Faithfulness on $J$ implies $x^{1/2}a=0$ for all $a\in J$, and
essentiality of $J$ gives $x=0$. Thus $M(J)$ has a faithful tracial
state and is stably finite.
\end{proof}

\begin{proof}[Proof of Theorem~\ref{thm:A}]
Take $J$ and $\varphi$ to be any of the examples constructed in Section~\ref{sec:examples}.
Theorem~\ref{thm:B} gives a separable unital subalgebra
$B\subseteq M(J)$ and a unital inclusion $\iota\colon C(D^3)\to B$
of real rank zero whose second amplification does not have real
rank zero. 
\end{proof}

Our examples do not address the following.

\begin{question*}
	If the inclusion $M_n(A)\subseteq M_n(B)$ has real rank zero for all $n \in \mathbb N$, does the inclusion $A\subseteq B$ approximately factor through real rank zero $C^*$-algebras?
\end{question*}

\end{document}